\documentclass[11pt]{amsart}
\usepackage{amsfonts}
\usepackage{graphicx}
\usepackage{amsmath}
\usepackage{amsthm}
\usepackage{amssymb}
\usepackage{amscd}
\usepackage{color}
\usepackage{url}

\usepackage{graphicx}
\usepackage{microtype}
\usepackage{enumitem}
\usepackage{pinlabel}
\usepackage[top=1.2in,bottom=1.4in,left=1.3in,right=1.3in]{geometry}

\theoremstyle{definition}

\newtheorem{theorem}{Theorem}[section]

\newtheorem*{proposition*}{Proposition}

\theoremstyle{definition}

\newtheorem{remark}[theorem]{Remark}

\newtheorem*{observation*}{Observation}

\newcommand{\N}{\mathbb{N}}

\newcommand{\calM}{\mathcal{M}}

\newcommand{\lesseq}{\preccurlyeq}

\DeclareMathOperator{\Homeo}{Homeo}
\DeclareMathOperator{\mcg}{Map}

\DeclareMathOperator{\Accu}{Accu}

\newcounter{notes}
\renewcommand{\paragraph}[1]{\medskip 
\noindent \textbf{#1}}

\title{Two results on end spaces of infinite type surfaces}
\author{Kathryn Mann and Kasra Rafi}
\date{}

\begin{document}

\begin{abstract}
We answer two questions about the topology of end spaces of infinite type surfaces and the action of the mapping class group that have appeared in the literature.  First, we
give examples of infinite type surfaces with end spaces that are not self-similar, but a unique maximal type of end, either a singleton or Cantor set.  Secondly, we use an argument of Tsankov to show that the ``local complexity" relation $\lesseq$ on end types gives an equivalence relation that agrees with the notion of being locally homeomorphic.  
\end{abstract}

\maketitle

\section{Introduction}
The paper \cite{largescale} introduced the notion of {\em self-similar} end spaces for infinite type surfaces, 
and proved that a self-similar end space necessarily contains a unique {\em maximal} type of end (with respect 
to the partial order defined below), with the set of ends of this type either a singleton or a Cantor set.  
This notion has turned out to be a useful one and it has appeared many times 
in the literature, see for instance \cite{BV, GRV, HQR, MT}.
A partial converse to this statement was proved in \cite[Prop. 4.8]{largescale}.
Namely, under the additional hypothesis that an infinite type surface $\Sigma$ contains no nondisplaceable 
subsurfaces, it was shown that the end space of $\Sigma$ is self-similar if and only if the set of maximal ends 
is either a singleton or a Cantor set of points of the same type.   However, the necessity of this extra hypothesis 
(no nondisplaceable subsurfaces) was not discussed there, raising the question of whether it could be eliminated.  
This appeared as Question 1.4 in \cite{LL}, and  Remark 6.2 in \cite{LV}.  Here we answer the question, 
showing the strict converse (without extra assumptions) to \cite[Prop. 4.8]{largescale} is false:  

\begin{theorem} \label{thm:main} 
There exist examples of surfaces that have non self-similar end spaces with a unique maximal type of end and set of maximal ends homeomorphic to a singleton or to a Cantor set.  
\end{theorem} 

The partial order defined in  \cite[Section 4]{largescale} is as follows. For $x,y \in E$, we say $y \lesseq x$ 
if every neighborhood of $x$ contains a point locally homeomorphic to $y$, and we say $x$ and $y$ are of the same type if $x \lesseq y$ and $y \lesseq x$, so that $\lesseq$ descends to a partial order on types.   Informally, we think of $\lesseq$ as describing the local complexity of an end.  Of course, if $h(x) = y$ for some homeomorphism $h$ of $E$, then $x$ and $y$ are of the same type.  Here we prove the converse to this statement, following an argument of T. Tsankov.  This answers another question from \cite{largescale}.  

\begin{theorem} \label{thm:order} 
If $x \lesseq y$ and $y \lesseq x$, then there is a mapping class (equivalently, a homeomorphism) $h$ of $\Sigma$ taking $x$ to $y$. 
\end{theorem} 

The outline of the paper is as follows.  First, we provide details on a local construction of non-comparable points (an idea sketched loosely in \cite{largescale}) as a tool towards the proof of Theorem \ref{thm:main}.  We then prove Theorem \ref{thm:main} building examples first in the case where the maximal end is a singleton, followed by the Cantor set case.  Finally, the proof of Theorem \ref{thm:order} is given in Section \ref{sec:order}.

\section{Toolkit for Theorem \ref{thm:main}: Non-comparable points} \label{Sec:non-comparable}

For a surface $\Sigma$, we denote the space of ends of $\Sigma$ by $E(\Sigma)$ 
or simply $E$.  We define an equivalence relation on the end space by saying points $x,x' \in E$ are \emph{locally homeomorphic} if there exists some clopen neighborhood of $x$ in $E$ that is 
homeomorphic to some clopen neighborhood of $x'$ via a homeomorphism taking $x$ to $x'$.  
This is equivalent to saying that there is a homeomorphism 
of $\Sigma$ such the the induced map on the end space sends $x$ to $x'$. For an end $x$, we let $\Accu(x)$ denote the set of accumulation points of the set of all ends locally homeomorphic to $x$.

We work with the relation $\lesseq$ on points of $E$ as given above in the introduction. One may equivalently define $y \lesseq x$ if $x \in \Accu(y)$.  We say $x$ and $y$ are {\em of the same type} if $x \lesseq y$ and $y \lesseq x$, so that $\lesseq$ descends to a partial order on types.  
We say $x$ is a {\em maximal type}
if $x \lesseq y$ implies $y \lesseq x$ and denote the set of maximal points in $E$ by $\calM(E)$.  
We say $x,y \in E$ are non-comparable if neither $x \lesseq y$ 
nor $y \lesseq x$ holds. See \cite{largescale} for more details and discussion. 

The first building block in our construction is a sequence of surfaces $D_n$ indexed by $n \in \N$, each
with one boundary component, such that $D_n$ contains a unique maximal end $z_n$ and for all 
$i \neq j$ the ends $z_i$ and $z_j$ are non-comparable (the reader should picture $D_i$ and $D_j$ 
as disjoint subsurfaces of $\Sigma$).

Note that a construction such as this is not possible when the surface is planar and 
has a countable number of ends. A classical result of Mazurkiewicz and Sierpinski 
\cite{MS} states that, for any surface with a countable set of ends, there exists a countable ordinal 
$\alpha$  such that the end space $E$ is homeomorphic to the ordinal 
$\omega^\alpha \cdot m + 1$ where $m$ is a positive integer. 
The assumption that $E$ has one maximal point implies that $m=1$.  
Now assume $D, D'$ are two genus zero surfaces with one boundary and a countable 
end space ($E$ and $E'$) such that that each end space has one maximal end 
($x\in E$ and $x'\in E'$). Then their end spaces are respectively homeomorphic to 
$\omega^\alpha + 1$ and  $\omega^{\alpha'}+ 1$ for some countable ordinals $\alpha$ and 
$\alpha'$. Now if $\alpha \leq \alpha$ then $x \lesseq x'$, which means $x$ and $x'$ 
are comparable. 

We carry out the construction in both remaining cases, namely, when the set of ends is uncountable and the surface is planar (the proof easily generalizes to non-planar surfaces), and when the set of ends is countable and the surface is non-planar.  

\subsection*{Uncountable planar case}
Let $D$ be a disc, let $C_n =Q_n \cup C \subset D$ be the union of a countable set $Q_n$ and 
a Cantor set $C$, with Cantor-Bendixson rank $n$ such that, for each derived set of $C_n$ that has
isolated points, the accumulation set of the isolated points contains the Cantor set.  
For example, one may take the $C$ to be the standard middle-thirds Cantor set, and insert in each missing interval a 
set homeomorphic to $\omega^n+1$ to form $Q_n$.   Now for each $C_n$, select a single point $z_n$ 
and let $C'_n$ be another Cantor set contained in $D$ so that $C_n \cap C'_n = \{z_n\}$.   Puncturing $D$ 
along $C_n \cup C'_n$ gives a surface $D_n$ with one boundary component such that $z_n$ is the unique 
maximal end.   By construction, $z_i$ and $z_j$ are non-comparable when $i \neq j$.  

\subsection*{Countable non-planar case} 
Let $D$ be a disk and let $\alpha$ and $\beta$ be two countable ordinals with $\beta < \alpha$. 
Let $E_\alpha$ be a subset of $D$ homeomorphic to $\omega^\alpha + 1$ and denote its (unique) maximal point by $z_{\alpha, \beta}$. Now, consider a  closed subset $E_\beta \subset E_\alpha$ homeomorphic to 
$\omega^\beta+1$ where $z_{\alpha,\beta}$ is again the maximal point of $E_\beta$. 
For every isolated point $y$ of $E_\beta$ remove a disk around $y$ (keeping these disks pairwise disjoint) and glue back 
in a one-ended, infinite genus surface with one boundary component. We also puncture 
$D$ along the remaining points of $E_\alpha$ to obtain a surface $D_{\alpha, \beta}$. The point 
$z_{\alpha, \beta}$ is the unique maximal end of this surface.  Moreover, for two pairs or countable ordinals $(\alpha, \beta)$ 
and $(\alpha', \beta')$ (satisfying $\beta < \alpha$ and $\beta' < \alpha'$) 
if $\alpha \geq \alpha'$ and $\beta < \beta'$ then $D_{\alpha, \beta}$ and 
$D_{\alpha', \beta'}$ are non-comparable. In fact, no end of $D_{\alpha, \beta}$ is of the same type
as $z_{\alpha', \beta'}$ and vice versa. Hence we can, for example, fix $\alpha$ and vary 
$\beta$ to get a countable family of surfaces with one boundary where the maximal points are 
non-comparable. 

\subsection*{Uncountably many non-comparable points}
It is also possible for a surface to contain uncountably many non-comparable points. 
For example, let $\Sigma$ be a sphere minus a Cantor set. Visualize $\Sigma$
as a union of pairs of pants. Enumerate the pairs of pants, remove a disk from each pair of pants, 
and glue back in a copy of $D_n$ to the $n$-th pair of pants.  Call the resulting surface $\Sigma'$. 
Then all the ends of $\Sigma'$ coming from 
$\Sigma$ are non-comparable since small enough neighborhoods of any two such ends 
contain non-comparable points. 

\section{Proof of Theorem \ref{thm:main}}

Now we construct the surface that will furnish the examples needed for Theorem \ref{thm:main}. 
We give the construction first for the case where $\calM(E)$ is a singleton. We then modify the 
construction to produce examples where $\calM(E)$ is a Cantor set.   

Start with a flute surface, meaning a sphere punctured along a sequence of points $p_1, p_2, ... $ accumulating 
at an end $p_\infty$.   For each $i \neq \infty$, replace a neighborhood of the puncture  $p_i$ with a Cantor 
tree $T_i$.  We think of $T_i$ as a union of pants surfaces, indexed by finite binary strings, so that the pants 
indexed by a string $s_1... s_n$ has cuffs glued to the pants indexed by $s_1 ... s_{n-1}$, $s_1 ... s_n 0$, and 
$s_1 ... s_n 1$, and the first pair of pants $P_\emptyset$ is glued on where the original puncture was removed.   
Now for each $i$, we will replace a countable set of discs in $T_i$  with discs homeomorphic to copies of the 
previously constructed discs $D_n$ (from either construction in the previous section), according to the following recipe. 

For tree $T_i$, place copies of $D_1, D_2, \ldots D_i$ on the first pants surface, $P_\emptyset$, and place a copy of $D_k$ on each pants indexed by a word of length $k-i$.   Thus, for each $k\geq i$ there are $2^{k-i}$ copies of $D_k$ on $T_i$.    Call the resulting punctured surface $S$.  An illustration is given in Figure \ref{fig:one_end}.  

  \begin{figure*}[h]
   \labellist 
  \footnotesize \hair 2pt
     \pinlabel $D_1$ at 315 68 
       \pinlabel $\ldots$ at 315 100 
         \pinlabel $D_i$ at 315 128 
      \pinlabel {\normalsize $T_i$} at 260 105 
      \pinlabel $D_{i+1}$ at 273 171 
      \pinlabel $D_{i+1}$ at 354 171
   \endlabellist
     \centerline{ \mbox{
 \includegraphics[width = 5.5in]{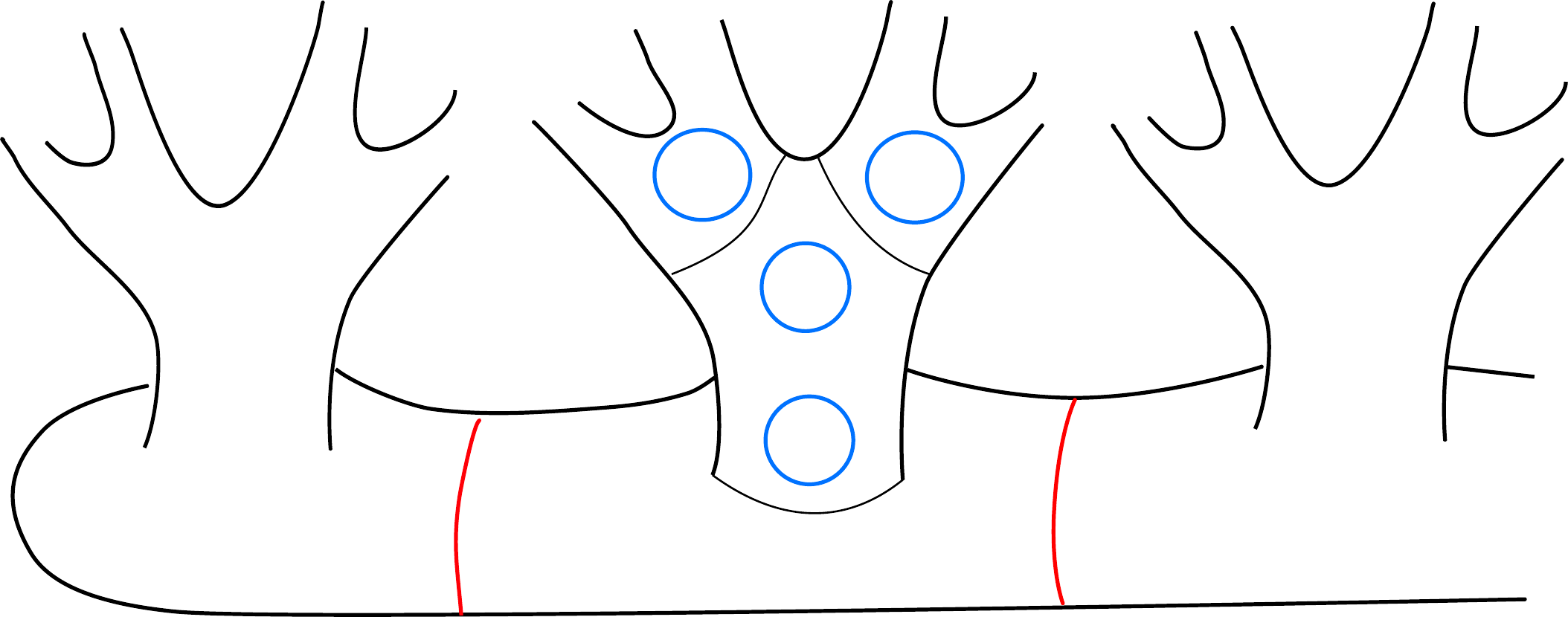}}}
 \caption{Construction of the surface with unique maximal end.}
  \label{fig:one_end}
  \end{figure*}
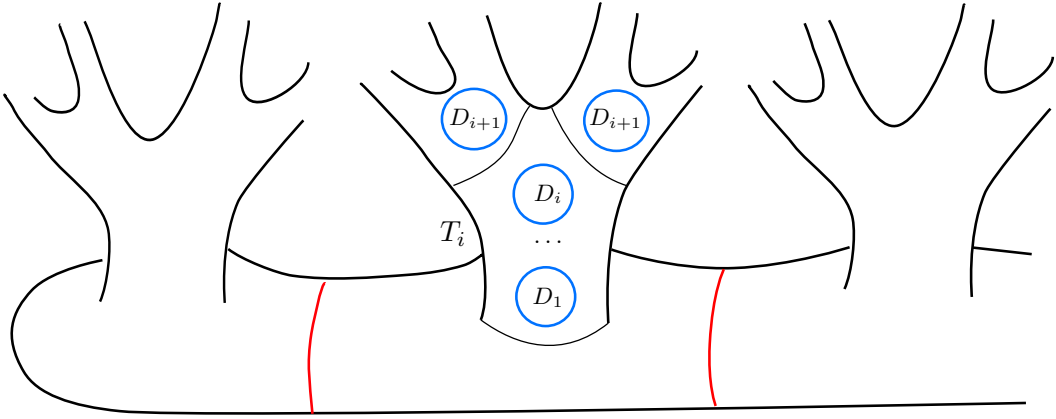

Note that all of the ends of each of the trees $T_i$ are pairwise locally homeomorphic.  The end $p_\infty$ of our surface $S$ is the unique accumulation point of these tree ends, so it is the unique maximal end.  We will now show that the end space of $S$ is not self-similar.   Let $E_i$ denote the end space of the tree $T_i$.  

Consider the decomposition of the end space $E_1 \sqcup (E - E_1)$.   Since $E_1$ does not contain a maximal end, to show the end space of the surface is not self-similar, it suffices to show that its complement contains no homeomorphic copy of $E_1$.   Suppose for contradiction that we could find such.  
Note the the sets $U_i := \bigcup_{n=i}^\infty E_n$ form a neighborhood basis of $p_\infty$ in the end space.  
Since $p_\infty$ is the unique maximal end and $E_1$ is closed, any homeomorphic copy of $E_1$ must avoid some neighborhood of $p_\infty$ so is contained in a finite union $E_2 \cup E_3 \cup \ldots \cup E_N$.   

By construction, $E_1$ contains $2^{N}$ locally homeomorphic copies of the end $z_{N+1}$.  But 
$E_2 \cup E_3 \cup \ldots \cup E_N$ contains $\sum_{i=1}^{N-1} 2^i < 2^N$ copies of $z_{N+1}$.  A contradiction.  Thus, $E_1$ cannot be mapped into its complement, so the end space is not self-similar.  

\subsection*{Cantor set case} 
A variation on the construction above can be used to produce a non self-similar surface with a unique maximal type and a Cantor set of maximal ends.  
First, following a similar procedure to the construction of the punctured trees $T_i$, for each $i \in \N$ we can build a Cantor tree $T'_i$ with a single boundary component that contains one copy of each of the discs $D_1$, $D_2$, ... $D_i$, and for each $k>i$ contains $2^{(2^{k-i})}$ copies of $D_k$, with each end of the tree locally homeomorphic.    

Now instead of starting with the flute, start with a Cantor tree constructed of pairs of pants indexed by binary strings, 
with the first pair of pants $P_\emptyset$ capped off by a disc on one of its boundary components.  From each pair of 
pants indexed by a string of length $i$, remove a disc and glue in a copy of $T'_{i+1}$ to it along its boundary. In particular, $T'_1$ is glued to the first pair of pants indexed by the empty string.  In the resulting surface, the ends of the original Cantor tree are precisely the maximal ends, forming a Cantor set of maximal ends of a single type.  We claim again that this is not self-similar.   To see this, let $E_1$ denote the end space of $T'_1$ and consider the decomposition of its end space into $E_1 \sqcup (E-E_1)$.   Suppose for contradiction that $E-E_1$ contained a homeomorphic copy of $E_1$.  As before, since $E_1$ and the set of maximal ends are both closed, the homeomorphic image of $E_1$ avoids some neighborhood of the maximal ends, so is contained in a union of end spaces of trees homeomorphic to $T'_i$ for a {\em bounded} set of indices $i$.   We consider the maximal such index $N$, and again count copies of ends of type $z_{N+1}$.  Without loss of generality, we may take $N\geq 4$.  
The set $E_1$ contains $2^{(2^{N})}$ copies of $z_{N+1}$.  Since our surface is constructed using $2^{k-1}$ copies of each tree $T'_k$, the number of copies of $z_{N+1}$ in the union of all trees $T'_k$ for $2 \leq k \leq N$ is equal to 
\[ 2 \cdot 2^{(2^{N-1})} + 2^2 \cdot 2^{(2^{N-2})} + \ldots + 2^{N-1}\cdot 2^{(2^{N-N+1})} \]
Set $j = 2^{N-1} + 1$. 
Then this sum is bounded above by 
\[ 2^j + 2^{j-1} + \ldots + 2^{j-N+2} < 2^{(2^N)} \] 
which gives the desired contradiction.

\section{Proof of Theorem \ref{thm:order}} \label{sec:order}

We now give the proof of Theorem \ref{thm:order}, following an argument of T. Tsankov.  The key ingredient is the following zero-one law for Baire sets invariant under certain actions of Polish groups.  

\begin{theorem} \label{thm:zeroone}(See \cite[Theorem 8.46]{Kechris}.)  
Let $G$ be a group of homeomorphisms of a Baire space $X$, and assume that for all open $U, V \subset X$ there exists $g \in G$ with $gU \cap V \neq \emptyset$.   Suppose that $A \subset X$ is a $G$-invariant set with the {\em Baire property}, meaning it differs from an open set (in the sense of symmetric difference) by a meager set.  
Then $A$ is either meager or has meager compliment in $X$.   
\end{theorem}

\begin{proof}[Proof of Theorem \ref{thm:order}]
Suppose $x \lesseq y$ and $y \lesseq x$.  Let $H_x$ denote the ends that are locally homeomorphic to $x$ and $H_y$ the ends locally homeomorphic to $y$.  If $H_x$ is finite, then it is easy to see that $H_x = H_y$.  Otherwise, we have $\overline{H_x} = \overline{H_y}$, and every point of $\overline{H_x}$ is an accumulation point, therefore $\overline{H_x} \subset E$ is homeomorphic to a Cantor set.  Denote this cantor set by $C$. This set $C$ is preserved by the action of $\mcg(\Sigma)$ on the end space.   

We claim the following: if $z \in C$ has a dense orbit in $C$ under $\mcg(\Sigma)$, then this orbit is comeager in $C$.  Thus, there is only one such orbit, showing that $x$ is locally homeomorphic to $y$.  
That is, the claim proves the Theorem. 

To do this, let $G \subset \Homeo(C)$ denote the quotient of $\mcg(\Sigma)$ defined by restricting the action to $C$.  Since $\mcg(\Sigma)$ is a Polish group and the kernel of the restriction map is closed, so is its quotient $G$.  Furthermore, the action of $G$ on $C$ has the topological transitivity property that for every open $U$ and $U'$ in $C$, there exists some $g \in G$ with $gU \cap U' \neq \emptyset$, since both $U$ and $U'$ intersect $H_x$.  Finally, any orbit of $G$ is a set with the Baire property, being the continuous image of a Polish space, so one may thus apply the topological zero-one law (Theorem \ref{thm:zeroone}) and conclude that a non-meager orbit is necessarily comeager in $C$.  

Thus, it suffices to show that any point $z$ with a dense orbit has a non-meager orbit.  
By a condition of Kechris-Rosendal \cite[Prop 3.2]{KR}, for this it suffices to show that for any open subgroup $V$ of $G$, the orbit $Vz$ is 
somewhere dense in $C$.  Let $V$ be an open subgroup.  Without loss of generality, we may take $V$ to be the 
subgroup consisting of homeomorphisms preserving a fixed decomposition of $C$ into finitely many clopen sets 
$C = X_1 \sqcup \ldots \sqcup X_n$ (possibly permuting the clopen sets) since such open subgroups form a 
basis for the topology of the homeomorphisms of the Cantor set.  Again, without loss of generality, assume 
$z \in X_1$. We will show that $Vz$ is dense in $X_1$.  Given some open set $U \subset X_1$, there exists 
$g \in G$ such that $gz \in U$.  We assume also $gz \neq z$.  Let $W$ be a clopen neighborhood of $z$ in 
$E$ small enough $W \cap C \subset X_1$, $gW \cap C \subset X_1$ and $gW \cap W = \emptyset$. Now 
define a homeomorphism $h$ of $E$ to agree with $g$ on $W$, agree with $g^{-1}$ on $gW$, and restrict 
to the identity elsewhere.  Then the restriction of $h$ to $C$ (i.e. the image of $h$ in $G$) lies in $V$, and 
$h(z) = g(z) \in U$.  This is what we needed to show.  
\end{proof} 

\begin{remark}
This proof came about as a response to the question: {\em does a homeomorphism $h$ of 
$\{0,1\}^{\mathbb{Z}}$ that set-wise preserves each periodic orbit of the full shift necessarily preserve 
all orbits of the shift?}   Tsankov \cite{Tsankov} answered this question in the negative, showing that, as in the 
proof above, there is only a single dense orbit under the group of periodic-orbit preserving maps of 
$\{0,1\}^{\mathbb{Z}}$.  The argument above is a direct translation to this setting.  
\end{remark}


\subsection*{Acknowledgements}
K.M. was partially supported by NSF grant DMS-1844516. K.R was partially supported by NSERC Discovery grant RGPIN 06486. We would like to thank Jing Tao, Sanghoon Kwak, and Nicholas Vlamis for comments on an earlier version of this paper, and Todor Tsankov for furnishing the proof of Theorem \ref{thm:order}.


\begin{thebibliography}{10}

\bibitem{BV}
A.~Bar-Natan and Y.~Verberne.
\newblock The grand arc graph
\newblock arXiv:2110.14761 [math.GT]

\bibitem{GRV}
C.~Grant and K.~Rafi and Y.~Verberne
\newblock Asymptotic Dimension of Big Mapping Class Groups 
\newblock arXiv:2110.03087 [math.GT]

\bibitem{HQR}
C.~Horbez and Y.~Qing and K.~Rafi.
\newblock Big mapping class groups with hyperbolic actions: Classification and applications. 
\newblock Journal of the Institute of Mathematics of Jussieu, page 1--32, 2021.

\bibitem{Kechris}
A.~S.~Kechris.
\newblock Classical descriptive set theory.
\newblock Springer GTM 156, 1995.

\bibitem{KR}
A.~S.~Kechris and C.~Rosendal.
\newblock Turbulence, amalgamation and generic automorphisms of homogeneous structures.
\newblock Proceedings of the London Mathematical Society. 94(2), 302-350, 2007.


\bibitem{LL}
J.~Lanier and M.~Loving.
\newblock Curve graphs of surfaces with finite-invariance index 1.
\newblock arXiv:2106.14916 [math.GT].

\bibitem{LV} 
J.~Lanier and N.~Vlamis.
\newblock  Mapping class groups with the Rokhlin property
\newblock arXiv:2105.11280 [math.GT].

\bibitem{largescale}
K.~Mann and K.~Rafi.
\newblock Large scale geometry of big mapping class groups.
\newblock arXiv:1912.10914 [math.GT].

\bibitem{MT}
J.~Malestein and J.~Tao
\newblock Self-similar surfaces: involutions and perfection
\newblock arXiv:2106.03681 [math.GT]

\bibitem{MS}
S.~Mazurkiewicz and W.~Sierpinski. 
\newblock Contribution \'a la topologie des ensembles d\'enombrables. 
\newblock Fundamenta Mathematicae, 1, 17--27, 1920.

\bibitem{Tsankov}
T.~Tsankov.
\newblock Personal communication. 

\end{thebibliography}
\end{document}